\documentclass[11pt,a4paper]{article}

\usepackage{amssymb,amsmath,amsthm}
\usepackage{authblk}

\usepackage{mathrsfs}
\usepackage{graphicx}
\usepackage[colorlinks,citecolor=blue,linkcolor=blue,urlcolor=red,pagebackref]{hyperref}
\usepackage{subcaption}
\usepackage{booktabs}
\usepackage{lineno}
\usepackage{color}
\usepackage{enumitem}
\usepackage{tikz}
\usepackage[linesnumbered,lined,boxed]{algorithm2e}
\usepackage{framed}
\usepackage{thm-restate}

\newcommand*\patchAmsMathEnvironmentForLineno[1]{%
  \expandafter\let\csname old#1\expandafter\endcsname\csname #1\endcsname
  \expandafter\let\csname oldend#1\expandafter\endcsname\csname end#1\endcsname
  \renewenvironment{#1}%
     {\linenomath\csname old#1\endcsname}%
     {\csname oldend#1\endcsname\endlinenomath}}%
\newcommand*\patchBothAmsMathEnvironmentsForLineno[1]{%
  \patchAmsMathEnvironmentForLineno{#1}%
  \patchAmsMathEnvironmentForLineno{#1*}}%
\AtBeginDocument{%
\patchBothAmsMathEnvironmentsForLineno{equation}%
\patchBothAmsMathEnvironmentsForLineno{align}%
\patchBothAmsMathEnvironmentsForLineno{flalign}%
\patchBothAmsMathEnvironmentsForLineno{alignat}%
\patchBothAmsMathEnvironmentsForLineno{gather}%
\patchBothAmsMathEnvironmentsForLineno{multline}%
}

\newtheorem{theorem}{Theorem}
\newtheorem{definition}[theorem]{Definition}
\newtheorem{corollary}[theorem]{Corollary}

\newtheorem{lemma}[theorem]{Lemma}

\newbox\ProofSym
\setbox\ProofSym=\hbox{%
\unitlength=0.18ex%
\begin{picture}(10,10)
\put(0,0){\framebox(9,9){}}
\put(0,3){\framebox(6,6){}}
\end{picture}}

\begin{document}

\title{A Sharp Upper Bound for the Frobenius Number}
\author{Lin Chen}
\author{Tingwei Hu}
\author{Yuchen Mao}
\author{Guochuan Zhang}
\affil{College of Computer Science and Technology\\
Zhejiang University\\
\{chenlin198662, tingweihu, maoyc, zgc\}@zju.edu.cn}
\date{}
\maketitle

\begin{abstract}
Let $0<a_1<\cdots<a_n$ be integers with $n\ge 2$ and
$\gcd(a_1,\ldots,a_n)=1$. Their Frobenius number
$F(a_1,\ldots,a_n)$ is the largest integer that cannot be expressed as a
nonnegative integer combination of $a_1,\ldots,a_n$.
Erd\H{o}s and Graham~\cite{EG72} proved the classical bound
$F(a_1,\ldots,a_n)\le
2a_{n-1}\lfloor a_n/n\rfloor-a_n$, which depends on the two largest
generators. Dixmier~\cite{D90} subsequently proved their conjecture
$F(a_1,\ldots,a_n)\le a_n^2/(n-1)$, a bound depending only on $a_n$ and
$n$ that removes the factor $2$ from the corresponding quadratic estimate.
Following the Erd\H{o}s--Graham approach, we remove the factor $2$
while retaining the dependence on $a_{n-1}$, and prove
\[
    F(a_1,\ldots,a_n)
    \le a_{n-1}\left\lceil\frac{a_n}{n-1}\right\rceil-a_n-1.
\]
This bound can be substantially smaller than Dixmier's when
$a_{n-1}$ is small relative to $a_n$. For every $n\ge3$, we also
construct infinitely many instances satisfying the reverse
inequality, which shows that the bound is sharp.
\end{abstract}

\section{Introduction}\label{sec:introduction}
Let $S=\{a_1,\ldots,a_n\}$, where
$0<a_1<\cdots<a_n$, $n\ge2$, and $\gcd(a_1,\ldots,a_n)=1$.
The numerical semigroup generated by $S$ is
$\langle S\rangle
=\{\sum_{i=1}^n x_i a_i\mid x_i\in\mathbb Z_{\ge0}\}$.
Its complement in $\mathbb Z_{\ge0}$ is finite, and the Frobenius
number $F(S)$ is its largest element (with $F(S)=-1$
if $a_1=1$). The problem is to determine or bound $F(S)$
from the generators; see~\cite{RA05}. 

For two generators, Sylvester's
formula gives the exact answer
\begin{equation}\label{eq:sylvester}
    F(a_1,a_2)=a_1a_2-a_1-a_2.
\end{equation}
For three or more generators, Erd\H{o}s and Graham~\cite{EG72} proved the bound
\begin{equation}\label{eq:erdos-graham}
    F(a_1,\ldots,a_n)
    \le 2a_{n-1}\left\lfloor\frac{a_n}{n}\right\rfloor-a_n.
\end{equation}
Fixing the number $n$ of generators and an upper bound $t$ on the
largest generator, Erd\H{o}s and Graham also considered the worst-case
Frobenius number
\[
    g(n,t)=\max\{F(B)\mid
       B\subseteq\{1,\ldots,t\},\
       |B|=n,\ \gcd(B)=1\}.
\]
They obtained the explicit estimates
\[
    \frac{t^2}{n-1}-5t\le g(n,t)<\frac{2t^2}{n}.
\]
The gap led to the conjecture that, for fixed $n$ as $t\to\infty$,
$g(n,t)\sim t^2/(n-1)$; this is recorded as Erd\H{o}s
Problem~\#433~\cite{EG80,EP433}. Dixmier~\cite{D90} resolved the
conjecture. His theorem and the arithmetic-progression example imply
the following concise bounds.\footnote{For simplicity, we use a slightly
weaker form of Dixmier's upper bound. For the full bound,
see~\cite{D90,HP05}.}
\[
    \left\lfloor\frac{t-2}{n-1}\right\rfloor(t-n+1)-1
    \le g(n,t)
    \le \left\lfloor\frac{t-1}{n-1}\right\rfloor(t-2)-1.
\]
The full bound is attained when $t\equiv0,1,2\pmod{n-1}$;
see~\cite{H98,HP05} for the extremal constructions and a restatement
of that theorem. Dixmier also proved the simpler universal bound
$F(S)\le a_n^2/(n-1)$~\cite{D90}.

The extremal problem above retains only $n$ and the upper bound $t$
on the largest generator. Yet both the two-generator formula
\eqref{eq:sylvester} and the Erd\H{o}s--Graham inequality
\eqref{eq:erdos-graham} involve the two largest generators. This
suggests a different question.
\begin{quote}\itshape
Can one obtain a sharper universal bound by keeping both $a_n$ and
$a_{n-1}$ as parameters?
\end{quote}

\subsection{Our contributions}
Our first result answers this question with the following bound.
\begin{theorem}[Upper bound]\label{thm:upper}
For every coprime tuple $0<a_1<\cdots<a_n$ with $n\ge2$,
\begin{equation}\label{eq:main-bound}
    F(a_1,\ldots,a_n)
    \le a_{n-1}\left\lceil\frac{a_n}{n-1}\right\rceil-a_n-1.
\end{equation}
\end{theorem}

\begin{theorem}[Lower bound]\label{thm:lower}
For each $n\ge3$, there are infinitely many instances
$0<a_1<\cdots<a_n$ such that
\[
    F(a_1,\ldots,a_n)
    \ge a_{n-1}\left\lceil\frac{a_n}{n-1}\right\rceil-a_n-1.
\]
\end{theorem}

The bound also has a density interpretation. If
$S_i=\{a_1,\ldots,a_i,a_n\}$ has greatest common divisor one,
then applying Theorem~\ref{thm:upper} to $S_i$ gives
\begin{equation}\label{eq:prefix-bound}
    F(S)\le
    \min_{\substack{1\le i\le n-1\\
                    \gcd(a_1,\ldots,a_i,a_n)=1}}
    \left\{a_i\left\lceil\frac{a_n}{i}\right\rceil-a_n-1\right\}.
\end{equation}
For each $i$ with $\gcd(a_1,\ldots,a_i,a_n)=1$, the ratio $i/a_i$
is the density of the first $i$ generators in $\{1,\ldots,a_i\}$.
Ignoring the ceiling correction, the controlling factor in
\eqref{eq:prefix-bound} is $a_i/i$, the reciprocal of this density.
Thus we can choose the prefix with the largest density among those
satisfying the gcd condition. In particular, the second-largest
generator can be deleted whenever the remaining set is still coprime.

\subsection{Related works}
Put $A=S\cup\{0\}$ and $t=a_n$, so $|A|=n+1$.
Lev's estimates for iterated sumsets~\cite{L96} imply the explicit
cardinality bound
\[
    |jA|\ge n+1+\sum_{k=2}^{j}\min\{t,k(n-1)+1\}
    \qquad (j\ge2).
\]
This provides another route to the Erd\H{o}s--Graham extremal problem.
Hamidoune~\cite{H98} went further in determining which sets actually
maximize the Frobenius number when $n$ and $t$ are fixed. For
$t\ge n(n-1)+2$, he determined this maximum and classified the
maximizing sets. For example, if $n\ge6$, $q\ge3$, and
$t=(n-1)q$, the unique maximizer among sets of size $n$ with largest
element $t$ is $\{q,2q,\ldots,(n-1)q,t-1\}$; its Frobenius number is
$(q-1)(t-2)-1$.

Hamidoune and Plagne~\cite{HP05} explained the arithmetic structure
behind slow sumset growth. If $j\ge2$ and
$|jA|\le j(j+1)n/2$, then either $n\ge t/j$, or $A$ is a union of
two arithmetic progressions with a common difference, a progression
modulo $t$, or, for even $t$, a set of the form
$\{0,t/2,t,x,x+t/2,2x\}$ with $0<x<t/2$. Thus, when $A$ is sparse
in $\{0,\ldots,t\}$, a small $j$-fold sumset forces a specific
arithmetic pattern. 

\subsection{Technical overview}
The central additive tool in Erd\H{o}s and Graham's proof is
Kneser's theorem~\cite{Kn53}. We use the following finite-group form.
\begin{theorem}[Kneser's theorem]\label{thm:kneser-intro}
Let $X,Y$ be nonempty finite subsets of an abelian group $G$, and put
$H=\{h\in G:X+Y+h=X+Y\}$, the symmetry set (or stabilizer) of
$X+Y$. Then
\[
    |X+Y|\ge |X+H|+|Y+H|-|H|.
\]
\end{theorem}
Work in $G=\mathbb Z_{a_n}$ with
$A=\{0,a_1,\ldots,a_{n-1}\}$, and write $kA$ for its $k$-fold
sumset. The aim is to cover all residue classes modulo $a_n$ while
controlling the sizes of their integer representatives. Set
$k_0=\lceil a_n/(n-1)\rceil$, which is roughly $a_n/n$. If $k_0A$
were proper and had trivial symmetry set, Kneser's theorem would give
$|k_0A|\ge k_0(n-1)+1>a_n$, a contradiction. Thus the only case
requiring further work is a proper $k_0A$ with a nontrivial symmetry
set $H$.

Erd\H{o}s and Graham instead take $k=\lfloor a_n/n\rfloor$,
show that $kA$ contains more than half the residues, and infer that
$2kA=G$. Bounding every summand by $a_{n-1}$ produces the factor
$2$ in~\eqref{eq:erdos-graham}. 

We therefore pass to $G/H$ and let $B$ be the image of $A$. Since
$k_0A$ is proper and $H$-periodic, $k_0B$ does not cover this smaller
group. Write $c=|G/H|-k_0(|B|-1)\ge2$. Kneser's theorem shows that at most $c-1$
quotient residues are still missing. We can cover them with at most
$c-1$ further copies of $B$. The point is that a larger $c$ also
gives us cheaper summands. Indeed, in the proof of
Lemma~\ref{lem:weak-upper} we show that, if $h=|H|$ and
$\Delta=|A+H|-|A|$ counts the empty positions in the $H$-cosets
met by $A$, then $(\Delta+1)/h\le1-c/k_0$. This means that a
larger $c$ forces fewer empty positions in these cosets. Each
nonzero class of $B$ therefore has more representatives in $A$,
and its least representative is at most $(1-c/k_0)a_{n-1}$.
When $c$ is
small, a representative may be larger, but few extra summands are
needed. In this way, we keep the total value at most $k_0a_{n-1}$.
Finally, a unique-representation sumset inequality of Scherk and
Kemperman~\cite{S55,K60} shows that $k_0a_{n-1}-a_n$ itself is
representable. Since the preceding argument bounds the Frobenius
number by this integer, we conclude that
$F(S)\le k_0a_{n-1}-a_n-1$. 

\section{Preliminaries}\label{sec:preliminaries}
\begin{definition}[Symmetry set]\label{def:symmetry-set}
For a subset $S$ of a finite abelian group $G$, its symmetry set is
\[
    \operatorname{Sym}_G(S)=\{h\in G\mid S+h=S\}.
\]
This subgroup is also called the \emph{stabilizer} of $S$. We say
that $S$ is $H$-periodic when $S+H=S$ for a subgroup $H\le G$.
\end{definition}
We will use the following iterated form of Kneser's theorem~\cite{Kn53}.
\begin{theorem}\label{coro:kneser2}
Let $S_1,\ldots,S_m$ be nonempty subsets of a finite abelian group
$G$, and put $S=S_1+\cdots+S_m$ and $H=\operatorname{Sym}_G(S)$. Then
\[
    |S|\ge\sum_{i=1}^m|S_i|-(m-1)|H|.
\]
\end{theorem}
\begin{proof}
Put $T_j=S_1+\cdots+S_j$. Any period of $T_j$ remains a period
after adding the other sets, so $\operatorname{Sym}_G(T_j)\subseteq H$. The
two-set form of Kneser's theorem therefore gives
$|T_j|\ge |T_{j-1}|+|S_j|-|H|$ for $2\le j\le m$.
Adding these inequalities proves the result.
\end{proof}

\begin{lemma}[Kemperman--Scherk theorem~\cite{S55,K60}]\label{lem:kemperman-scherk}
Let $X,Y$ be finite nonempty subsets of an abelian group. Every
$z\in X+Y$ has at least $|X|+|Y|-|X+Y|$ representations $z=x+y$
with $x\in X$ and $y\in Y$.
\end{lemma}

\begin{corollary}\label{lem:Kem}
If some element of $X+Y$ has a unique representation as $x+y$,
then $|X+Y|\ge |X|+|Y|-1$.
\end{corollary}

\section{Proofs of the main results}\label{sec:bounds}
\subsection{Upper bound}

\begin{lemma}\label{lem:weak-upper}
For every coprime tuple $0<a_1<\cdots<a_n$ with $n\ge2$,
\[
    F(a_1,\ldots,a_n)
    \le a_{n-1}\left\lceil\frac{a_n}{n-1}\right\rceil-a_n.
\]
\end{lemma}
\begin{proof}
Put $G=\mathbb Z_{a_n}$, $A=\{0,a_1,\ldots,a_{n-1}\}$, and
$k=\lceil a_n/(n-1)\rceil$. We regard $A$ and its sumsets as subsets
of $G$. For $\rho\in G$, let $R(\rho)$ be the least nonnegative sum of
$a_1,\ldots,a_{n-1}$ congruent to $\rho$ modulo $a_n$. Such a sum
exists because the generators are coprime. In each residue class,
the representable integers start at $R(\rho)$; hence the last gap
there, when nonnegative, is $R(\rho)-a_n$. Thus
\[
    F(a_1,\ldots,a_n)=\max_{\rho\in G}R(\rho)-a_n.
\]
We first show that $R(\rho)\le k a_{n-1}$ for every $\rho$.

If $kA=G$, this is immediate, since every residue has a representative
using at most $k$ elements of $A$, which is at most $k a_{n-1}$. Suppose instead that $kA$ is
proper, and let $H=\operatorname{Sym}_G(kA)$. The subgroup $H$ is nontrivial. If it
were trivial, the iterated Kneser bound would give
$|kA|\ge k(n-1)+1>a_n$. Also, since $0\in kA$ and $kA$ is
$H$-periodic, we have $H\subseteq kA$.

Write $h=|H|$ and $d=a_n/h$. Thus $H$ consists of the multiples of
$d$ in $G$. Let $B$ be the image of $A$ in $G/H\cong\mathbb Z_d$.
The sumset $kB$ is proper and has trivial stabilizer; otherwise a
nonzero period of $kB$ would lift to a period of $kA$ outside $H$.
Theorem~\ref{coro:kneser2} now gives
$|kB|\ge k(|B|-1)+1$. Set
\[
    c=d-k(|B|-1).
\]
Since $kB$ is proper, we have $c\ge2$, and $|kB|\ge d-c+1$. Moreover,
$B$ generates $\mathbb Z_d$. If there is a common divisor of $d$ and all the elements of $B$, this common divisor would divide $a_1,\ldots,a_n$. Because $0\in B$,
adding $B$ strictly enlarges any proper sumset of $B$. Indeed, if
$S+B=S$, every element of $B$ is a period of $S$, so $S=\mathbb Z_d$.
Consequently, $|jB|\ge |(j-1)B|+1$ whenever $(j-1)B$ is proper.
Starting with $kB$, at most $c-1$ further additions therefore cover
the quotient,
\[
    (k+c-1)B=\mathbb Z_d.
\]

We next estimate the cost of lifting a sum from this quotient back
to $G$. The full preimage of $B$ is $A+H$, a union of $|B|$ cosets
with $h|B|$ elements. Since $A$ contains $n$ of them, the number
missing from these cosets is
$\Delta=|(A+H)\setminus A|=|A+H|-|A|=h|B|-n$.
Put $\alpha=(\Delta+1)/h$. Since
$k(n-1)\ge a_n=hd$ and
$n-1=h(|B|-1)+h-\Delta-1$, we obtain
\[
    k(h-\Delta-1)\ge hc.
\]
Therefore,
\[
    \Delta<h-1\qquad\text{and}\qquad
    \alpha=\frac{\Delta+1}{h}\le1-\frac ck.
\]

For $b\in B$, let $A_b$ be the elements of $A$ congruent to $b$
modulo $d$. Thus $\delta_b=h-|A_b|$ counts the missing elements
in the coset $b+H$, and $\sum_{b\in B}\delta_b=\Delta$. For
$b\ne0$, let $\mu_b=\min A_b$ be the smallest generator in that
coset. At most $\delta_b$ positions precede $\mu_b$
in its coset, so $\mu_b\le(\delta_b+1)d$. Among the
$h-\delta_b$ elements of $A_b$, the largest is at least
$\mu_b+(h-\delta_b-1)d$ and at most $a_{n-1}$. Combining these
two observations yields
\[
    \mu_b\le\frac{\delta_b+1}{h}a_{n-1}
          \le\alpha a_{n-1}\qquad(b\in B\setminus\{0\}).
\]
The gap bound also lets two elements of $A$ fill any nonzero coset
met by $A$. For $b\ne0$ we have
$|A_b|+|A_0|=2h-\delta_b-\delta_0\ge2h-\Delta>h$.
For any $x\in b+H$, the subsets $x-A_b$ and $A_0$ of $H$ must
therefore intersect. Hence
\[
    A_b+A_0=b+H.
\]

Now fix $\rho\in G$. If $\rho\in H$, the inclusion $H\subseteq kA$
already gives $R(\rho)\le k a_{n-1}$. Otherwise, use
$(k+c-1)B=\mathbb Z_d$ to express the image of $\rho$ as a sum
$b_1+\cdots+b_r$, where every $b_i$ is nonzero and
$1\le r\le k+c-1$. Here we discard any zero summands. Choose the
smallest generator $\mu_{b_i}$ in each coset $b_i+H$ for $i\ge2$,
and let $t=\sum_{i=2}^r\mu_{b_i}$ (with $t=0$ if $r=1$).
The quotient representation says that $\rho-t$ belongs to
$b_1+H$. By $A_{b_1}+A_0=b_1+H$, there are $x\in A_{b_1}$ and
$y\in A_0$ with $t+x+y\equiv\rho\pmod{a_n}$. This is a
nonnegative sum of the original generators. Each $\mu_{b_i}$ costs
at most $\alpha a_{n-1}$, while $x$ and $y$ each cost at most
$a_{n-1}$. Hence its total value is at most
$(2+\alpha(r-1))a_{n-1}$. By the bounds on $r$ and $\alpha$,
\[
    2+\alpha(r-1)
    \le2+\left(1-\frac ck\right)(k+c-2)
    =k-\frac{c(c-2)}{k}\le k.
\]
Thus $R(\rho)\le k a_{n-1}$ in every residue class, and so
$F(a_1,\ldots,a_n)\le k a_{n-1}-a_n$.
\end{proof}

\begin{proof}[Proof of Theorem~\ref{thm:upper}]
It suffices to prove that
$a_{n-1}\lceil a_n/(n-1)\rceil-a_n$ is representable.
Indeed, Lemma~\ref{lem:weak-upper} then gives
\[
    F(a_1,\ldots,a_n)
    \le a_{n-1}\left\lceil\frac{a_n}{n-1}\right\rceil-a_n-1.
\]
Put $k=\lceil a_n/(n-1)\rceil$,
$A=\{0,a_1,\ldots,a_{n-1}\}$, and $D=a_{n-1}-A$. This set has
$n$ elements, contains $0$, and lies in $[0,a_{n-1}]$. We claim that
some nonzero sum of at most $k$ elements of $D$ is divisible by
$a_n$. Otherwise, for each $1\le j\le k$, the zero residue has only
the all-zero representation in $jD$. Applying
Corollary~\ref{lem:Kem} to $(j-1)D+D$ for $j=2,\ldots,k$ gives
$|kD|\ge k(n-1)+1>a_n$, a contradiction.

Choose nonzero $d_1,\ldots,d_r\in D$, with $r\le k$, whose sum is
$s a_n$ for some $s\ge1$. Then
\[
    k a_{n-1}-a_n
    =(k-r)a_{n-1}
      +\sum_{i=1}^r(a_{n-1}-d_i)
      +(s-1)a_n.
\]
Every term on the right is a nonnegative combination of the original
generators. Thus $k a_{n-1}-a_n$ is representable, as required.
\end{proof}

\subsection{Lower bound}
\begin{proof}[Proof of Theorem~\ref{thm:lower}]
For each integer $q\ge2$, take
\[
    S_{n,q}=\{iq\mid 1\le i\le n-1\}\cup\{(n-2)q+1\}.
\]
The two generators $q$ and $(n-2)q+1$ are coprime, and every other
element of $S_{n,q}$ is a multiple of $q$. Hence
$\langle S_{n,q}\rangle=\langle q,(n-2)q+1\rangle$, and Sylvester's
formula gives
\[
    F(S_{n,q})=q\bigl((n-2)q+1\bigr)-q-\bigl((n-2)q+1\bigr)
    =\bigl((n-2)q+1\bigr)q-(n-1)q-1.
\]
Here $a_{n-1}=(n-2)q+1$, $a_n=(n-1)q$, and
$\lceil a_n/(n-1)\rceil=q$. Thus $F(S_{n,q})$ equals the right-hand
side of \eqref{eq:main-bound}.
\end{proof}

\section*{Acknowledgements}
The authors first proved the weaker upper bound in
Lemma~\ref{lem:weak-upper}. GPT-6 then helped find the argument that
lowers its right-hand side by one. GPT-6 also assisted with the
proof of the sharpness result in Theorem~\ref{thm:lower} and with
editing the English prose. The authors take responsibility for the
final manuscript.

\bibliographystyle{plain}
\bibliography{main2}
\end{document}